\documentclass [12pt] {amsart}

\usepackage [latin1]{inputenc}
\usepackage[all]{xy}
\usepackage[english]{babel}
\usepackage{amssymb}
\usepackage[lite, initials]{amsrefs}

\newtheorem{teorema}{Theorem}[section]
\newtheorem*{theorem*}{Structure Theorem}
\newtheorem{lemma}[teorema]{Lemma}
\newtheorem{propos}[teorema]{Proposition}
\newtheorem{corol}[teorema]{Corollary}
\theoremstyle{definition}

\newtheorem{rem}{Remark}[section]

\def\R{{\mathbb R}}
\def\N{{\mathbb N}}
\def\C{{\mathbb C}}

\newcommand{\Ci}{\mathcal{C}}

\def\Re{{\sf Re}}
\def\Im{{\sf Im}}

\def\Sing{{\rm Sing}}

\title{Domains with a continuous exhaustion in weakly complete surfaces}
\author[S.~Mongodi]{Samuele Mongodi\textsuperscript{1}}
\address{\textsuperscript{1}Politecnico di Milano, Dipartimento di Matematica, Via Bonardi, 9 -- I-20133 Milano, Italy}
\email{samuele.mongodi@polimi.it}
\author[Z.~Slodkowski]{Zbigniew Slodkowski\textsuperscript{2}}
\address[\textsuperscript{2}]{Department of Mathematics, University of Illinois at Chicago, 851 South Morgan Street, Chicago, Illinois 60607, Usa}
\email{zbigniew@uic.edu}

  \date{\today}
\subjclass[2010]{32C40, 32E05,32U10}

\begin{document}
\begin{abstract}
In previous works, G. Tomassini and the authors studied and classified complex surfaces admitting a real-analytic pluri\-subharmonic exhaustion function; let $X$ be such a surface and $D\subseteq X$ a domain admitting a \emph{continuous} plurisubharmonic exhaustion function: what can be said about the geometry of $D$? If the exhaustion of $D$ is assumed to be smooth, the second author already answered this question; however, the continuous case is more difficult and requires different methods. In the present paper, we address such question by studying the local maximum sets contained in $D$ and their interplay with the complex geometric structure of $X$; we conclude that, if $D$ is not a modification of a Stein space, then it shares the same geometric features of $X$.
\end{abstract}

\maketitle

\section{Introduction}

In \cites{crass,mst, mst2}, G. Tomassini and the authors initiated the study of the geometry of weakly complete spaces; the question had been raised by G. Tomassini and the second author in \cite{ST}, where they also defined the minimal kernel of a weakly complete space, an essential tool in the investigations presented in \cites{mst, mst2}.

In that first attempt, complex surfaces which admit a real-analytic plurisubharmonic exhaustion function were considered; in \cite{M}, the first author considered, inside such surfaces, open domains with an exhaustion which could be, \emph{a priori}, only smooth and proved that they actually admit a real-analytic exhaustion.

This paper aims to study the case of an open domain, inside a complex surface with a global real-analytic exhaustion, which admits only a continuous plurisubharmonic exhaustion function, thus generalizing the results of \cite{M} to the continuous case.

We obtain the following result.
\begin{teorema}\label{teo_main}Let $X$ be a complex surface which admits a real-analytic plurisubharmonic exhaustion function and consider a domain $D\subseteq X$ which admits a continuous plurisubharmonic exhaustion function. Then, one of the following holds true:
\begin{enumerate}
\item $D$ is a modification of a Stein space of dimension $2$
\item $X$ and $D$ are proper over open Riemann surfaces
\item $X$ and $D$ are of Grauert type.
\end{enumerate}\end{teorema}

Such an exploration is also useful as a small step towards the general problem of understanding the geometry of a weakly complete surface, with no assumptions on the regularity of the exhaustion.

\medskip

Our work is organized as follows.

In Section 2 we gather some preliminary results which are essentially well known in the smooth case and whose extension to the continuous case relies, for the most part, on the approximation results of Richberg (see \cite{rich}); Section 3 studies the connections between compact sets with the local maximum property and the (complex) analytic structure of the minimal kernel.
Finally, Section 4 contains the proof of our main result, together with some easy, but interesting, consequences.

\section{Preliminary results}

Let $M$ be a complex manifold.


An upper semicontinuous function $\phi:M\to\R$ is called \emph{strongly} plurisubharmonic at $p\in M$ if it is plurisubharmonic at $p\in M$ and stays so under small $\Ci^2$ perturbations; if this happens at every point $p\in U$ for some open set $U$, we call $\phi$ a strongly plurisubharmonic function on $U$.

Suppose $M$ is endowed with a continuous plurisubharmonic exhaustion function, i.e. $M$ is weakly complete; for every such function $\phi:M\to\R$, let $Z_\phi$ be the set of points $p\in M$ at which $\phi$ is \emph{not} strongly plurisubharmonic. We define the \emph{continuous minimal kernel} of $M$ as
$$\Sigma_M=\bigcap_{\phi}Z_\phi$$
where the intersection is taken over all the continuous plurisubharmonic exhaustion functions $\phi:M\to\R$. This is the same construction described in \cite{ST} for $\Ci^k$ functions, $k\geq 2$; we do not emphasize the use of continuous functions in the symbol $\Sigma_M$, as opposed to $\Ci^2$ o $\Ci^\infty$ or real-analytic functions, because in the present paper we only deal with the continuous case and, for the same reason, we will refer to it simply as the minimal kernel.

A continuous plurisubharmonic exhaustion function $\phi:M\to\R$ is called \emph{minimal} (for $\Sigma_M$) if $Z_\phi=\Sigma_M$; it can be showed that, if $M$ is weakly complete, a continuous minimal function always exists.

\medskip

Now, we present the generalization to the continuous case of the following well-known result: if a manifold admits a smooth exhaustion function which is strongly plurisubharmonic outside of a compact set, then it is a modification of a Stein space of the same dimension.

The proof is essentially a suitable application of Richberg approximation theorem (see \cite{rich}); the result itself is surely known, but we were not able to find an appropriate reference.

\begin{lemma}\label{lmm_richberg}Let $M$ be a complex manifold and $\phi:M\to\R$ a continuous plurisubharmonic exhaustion function. Assume that $\phi(\Sigma_M)$ does not contain any half-line $(a,+\infty)$, with $a\in\R$. Then $M$ is a modification of a Stein space.
\end{lemma}
\begin{proof}Let $\psi$ be a (continuous) minimal function on $M$.

Since $\phi(\Sigma_M)$ does not contain any half-line, there is an increasing sequence $\{c_n\}_{n\in\N}$ of real numbers such that $c_n\to+\infty$  and $\phi(\Sigma_M)\not\ni c_n$ for every $n\in\N$. Set $c_{-1}=-\infty$ and
$$F_n=\Sigma_M\cap\{c_{n-1}<\phi<c_n\}\;.$$
The sets $F_n$ are compact and pairwise disjoint and their union is $\Sigma_M$; we define
$$c'_n=\min_{F_n}\phi\qquad c''_n=\max_{F_n}\phi\;.$$
We have $c'_n<c_n<c''_{n+1}$, so
$$c'_0\leq c''_0<c'_1\leq c''_1<\ldots\leq c''_{n-1}<c'_n\leq c''_{n}<c'_{n+1}\leq \ldots$$
and $\phi(F_n)\subseteq [c'_n, c''_n]$.

Consider $U_n=\{c''_{n}<\phi<c'_{n+1}\}$, for a fixed $n$, and the function $\phi_t=\phi+t\psi$, where $t>0$ is chosen such that
$$t\max_{\overline{U}_n}|\psi|<\frac{c'_{n+1}-c''_{n}}{8}=\delta\;.$$
The function $\phi_t$ is strongly plurisubharmonic on $U_n$, because $\psi$ is minimal and $\Sigma_M$ does not intersect $U_n$; moreover, if $c''_n+2\delta<c<c'_{n+1}-2\delta$, then
$$\{\phi_t=c\}\subset \{|c-\phi|=t|\psi|\}\subseteq\{|c-\phi|<\delta\}\subseteq\{c-\delta<\phi<c+\delta\}\subset U_n$$
which implies that the level set $\{\phi_t=c\}$ is compact in $U_n$.

Set $V_n=\{c''_n+2\delta<\phi_t<c'_{n+1}-2\delta\}$ and apply Richberg's approximation result \cite{rich}*{Satz 4.3} to $\phi_t\vert_{V_n}$, which is strongly plurisubharmonic, in order to obtain a $\Ci^\infty$ function $\psi_n:V_n\to\R$ such that $\psi_n$ is strictly plurisubharmonic on $V_n$ and
\begin{equation}\label{eq_stimarich}
0\leq \psi_n(x)-\phi_t(x)\leq \frac{1}{2}\min\{\phi_t(x)-c''_n-2\delta, c'_{n+1}-2\delta-\phi_t(x)\}\;.\end{equation}

By Sard's theorem, there is $d_n\in (c''_{n}+2\delta, c'_{n+1}-2\delta)$ such that the level $\{\psi_n=d_n\}$ is regular; by \eqref{eq_stimarich}, such a level is compact. Therefore, the domain
$$\Omega_n=\{x\in M\ :\ \phi_t(x)<c''_n+2\delta\}\cup\{x\in V_n\ :\ \psi_n(x)<d_n\}$$
is relatively compact in $M$ and, for $0<\eta<d_n-c''_n-2\delta$, the function
$$\rho_n(x)=\left\{\begin{array}{ll}\phi_t(x)&x\in M\setminus (\Omega_n\cup V_n)\\\max\{c''_n+2\delta+\eta, \psi_n(x)\}&x\in V_n\\c''_n+2\delta+\eta&x\in\Omega_n\setminus V_n\end{array}\right.$$
is continuous and plurisubharmonic on $M$,smooth and strictly plurisubharmonic on a neighborhood of $b\Omega_n$. By \cite{grau1}*{Theorem 1}, $\Omega_n$ is holomorphically convex and, given the existence of $\rho$, a modification of a Stein space.

We note that $\Omega_n$ is strictly contained in $\{\phi<c'_{n+1}\}$ and $\Omega_{n+1}$ contains $\{\phi<c''_{n+1}\}$, so, as $c''_{n+1}\geq c'_{n+1}$, we have that $\Omega_n\Subset\Omega_{n+1}$.

Therefore $\Omega_1\Subset\Omega_2\Subset\ldots\Subset\Omega_n\Subset\ldots$ is an increasing family of relatively compact domains which are all modifications of Stein spaces and all have smooth stricly pseudoconvex boundary.

Let us show that $\Omega_j$ is Runge in $\Omega_{j+1}$, for every $j$; to such aim, we will use some well known facts about Remmert reduction, which we collected in \cite{Fuffa}, as we were not able to find a concise reference.

Let $\pi_j:\Omega_j\to Y_j$ be the Remmert reduction of $X_j$; let $R$ be the relation on $\Omega_j$ such that $x_1Rx_2$ if and only if there exists a compact connected complex subvariety of $\Omega_j$ which contains both $x_1$ and $x_2$, then $Y_j=X_j/R$. Therefore, $\pi_j^{-1}(y)$ is a compact complex subvariety of $\Omega_j$  for every $y\in Y_j$, so, $\rho_{j-1}$ is constant on $\pi^{-1}_j(y)$ for all $y\in Y_j$. This grants us the existence of a continuous function $\psi_j:Y_j\to \R$ such that $\psi_j\circ \pi=\rho_{j-1}$.

Set
$$S_j=\{y\in Y_j\ :\ \dim\pi^{-1}_j(y)>0\}\;,$$
then $\pi_j$ is a biholomorphism from $X_j\setminus \pi_j^{-1}(S_j)$ to $Y_j\setminus S_j$, therefore $\psi_j$ is plurisubharmonic on $Y_j\setminus S_j$ and, being continuous, is bounded around every point of $S_j$. By \cite{GraRem}*{Satz 3}, $\psi_j\vert_{Y_j\setminus S_j}$ extends uniquely to a plurisubharmonic function $\tilde{\psi}_j:Y_j\to\R$; let $y\in S_j$ and let $S_y=\pi_j^{-1}(y)$. We know that $\rho_{j-1}$ and $\tilde{\psi}_j\circ\pi_j$ are both constant on $S_y$; let $j:\Delta\to M$ be a complex disc intersecting $S_y$ at a single regular point $p\in S_y$, such that $j(0)=p$. The functions $\rho_{j-1}\circ j$ and $\tilde{\psi}_j\circ\pi_j\circ j$ are both subharmonic on $\Delta$ and coincide in $\Delta\setminus\{0\}$, but then
$$\tilde{\psi}_j\circ\pi_j\circ j(0)=\limsup_{\zeta\to 0} \tilde{\psi}_j\circ\pi_j\circ j(\zeta)=\limsup_{\zeta\to 0}\rho_{j-1}\circ j(\zeta)=\rho_{j-1}\circ j(0)\;.$$
So $\rho_{j-1}(p)=\tilde{\psi}_j\circ\pi_{j}(p)$, therefore $\tilde{\psi}_j(y)=\psi_j(y)$, thus proving that $\psi_j$ is (continous and) plurisubharmonic on $Y_j$.

In conclusion, by \cite{Nar}*{Corollary 1, Section 4}, the set
$$\{y\in Y_j\ :\ \psi_j(y)<d_{j-1}\}$$
is Runge in $Y_j$. By \cite{KK}*{Lemma 63.4}, this implies that $X_{j-1}$ is Runge in $X_j$.

\medskip

Finally, as $M$ is an increasing union of holomorphically convex domains, each one Runge in the next one, we conclude that $M$ is itself holomorphically convex. Let $\pi:M\to Y$ be the Remmert reduction of $M$; as $\psi$ is strongly plurisubharmonic on $M\setminus\Sigma_M$, we conclude that $\pi:M\setminus\Sigma M\to \pi(M\setminus\Sigma_M)$ is biholomorphism, therefore $Y$ and $M$ have the same dimension and $Y$ is an irreducible (because $M$ is a manifold) Stein space, so $M$ is a modification of a Stein space.

It is now easy to show that
$$S=\{y\in Y\ :\ \dim \pi^{-1}(y)>0\}$$
is a discrete set in $Y$, which implies that $M$ is the modification of $Y$ along a collection of at most countably many points.\end{proof}

As an obvious consequence, we have the following.

\begin{corol}Let $M$ be a complex manifold, with $\phi:M\to\R$ a continuous plurisubharmonic exhaustion, which is strongly plurisubharmonic outside a compact set. Then $M$ is a modification of a Stein space.\end{corol}

\begin{rem}
In the situation of Lemma \ref{lmm_richberg}, the Remmert reduction of $M$ has the same dimension of $M$, i.e. there exists a proper holomorphic function $f:M\to N$, where $N$ is a Stein space of the same dimension of $M$; moreover there exists a locally finite collection $\{E_j\}_{j\in \N}$ of compact complex subspaces of $N$, pairwise disjoint, such that, setting $E=\bigcup_j E_j$, the restriction of $f$ to $M\setminus E$ is a biholomorphism and $f(E)$ is a discrete set of points in $N$. It is easy to show that $\Sigma_M=E$.
\end{rem}

As a general principle, if we have an exhaustion which is strongly plurisubharmonic outside a compact set which is, in some sense, ``special'' for plurisubharmonic functions, then such a compact set has to belong to the minimal kernel; we see an easy instance of this idea in the following statement.

\begin{corol}\label{cor_modstein}Assume $M$ admits a continuous plurisubharmonic exhaustion function which is strictly plurisubhamonic outside $\bigcup_n C_n$, where $\{C_n\}$ is a sequence of compact complex curves. Then $\Sigma_M=\bigcup_n C_n$ and $M$ is a modification of a Stein space.\end{corol}
\begin{proof}By definition, $\Sigma_M\subseteq \bigcup_n C_n$. Moreover, given a minimal function $\alpha:M\to\R$, we see that $\alpha\vert_{C_n}\equiv c_n\in\R$, so $\Sigma_M=\bigcup_n C_n$; therefore, $\alpha(\Sigma_M)$ is a countable set in $\R$ and the previous lemma applies.\end{proof}

\section{Minimal kernel and local maximum sets}

The general principle stated before Corollary \ref{cor_modstein} applies also to compact sets with the local maximum property, as it is shown in the next lemma.
The rest of this section studies the interplay between local maximum sets and the analytic structure of the minimal kernel that was studied in \cite{mst}.

\begin{lemma}\label{lmm_complete}Let $K$ be a compact set in $M$ with the local maximum property and suppose there exists $\phi:M\to\R$ continuous plurisubharmonic exhaustion which is strongly plurisubharmonic on $X\setminus K$. Then $K$ is contained in the union of finitely many compact complex subspaces of $M$.\end{lemma}
\begin{proof}Obviously, $\Sigma_M\subseteq K$. By Lemma \ref{lmm_richberg} and the Remark following it, we have a proper holomorphic surjection $f:M\to N$, where $N$ is a Stein space and $f(\Sigma_M)$ is a discrete set; however, as $\Sigma_M$ is compact in this case, $f(\Sigma_M)$ will be a finite number of points. Therefore, $\Sing N$ consists of finitely many points and we can embedd properly $N$ into some $\C^n$ by some holomorphic map $j:N\to \C^n$. The set $S=j(f(\Sigma_M))$ is obviously complete pluripolar in $\C^n$, so we have a plurisubharmonic function $u:\C^n\to\R\cup\{-\infty\}$ such that $u^{-1}(-\infty)=S$; by considering $\rho=u\circ j\circ f$, we show that $\Sigma_M$ is complete pluripolar in $M$. We can suppose that $\rho$ is smooth and strongly plurisubharmonic outside $\Sigma_M$.

Now, suppose that $K\not\subseteq \Sigma_M$ and let $m=\max_{K}\rho$; then $m>-\infty$. Take $z_0\in K$ such that $m=\rho(z_0)$ and consider a small ball $B$ around $z_0$ such that $B\cap\Sigma_M=\emptyset$ and a function $\chi\in\Ci^\infty_0(B)$ such that
$$\chi\geq 0,\quad \chi(z_0)=1,\quad \chi(x)<1\textrm{ if }x\in B\setminus\{z_0\}\;.$$
For sufficiently small $\epsilon>0$, the function $\rho_\epsilon=\rho+\epsilon\chi$ is plurisubharmonic in $B$ and its restriction to $B\cap K$ has a strict maximum in $z_0\in K$, which contradicts the characterization of local maximum sets given in \cite{Sl2}*{Prop. 2.3(iv)}. As this is impossible, we conclude that $K=\Sigma_M$.\end{proof}

In \cite{mst}, we describe the structure of the minimal kernel of a complex surface which admits a real-analytic plurisubharmonic exhaustion; we find two types of compact sets with some analytic structure: compact complex curves and compact Levi-flat hypersurfaces with dense complex leaves. Both are ``minimal'' among compact sets with local maximum property, in a sense which we make more precise in what follows.

The next proposition is obvious.

\begin{propos}\label{prp_ovvia}Let $C$ be a complex compact irreducible curve in a complex manifold $M$. Then $C$ does not contain any proper compact local maximum set.\end{propos}

More generally, the presence of a Levi foliation prescribes the behaviour of local maximum sets.

\begin{lemma}Let $M$ be a complex manifold of (complex) dimension $2$ and $S\subset M$ a compact regular real-analytic Levi-flat hypersurface (i.e. of real dimension $3$). Let $K$ be a compact local maximum set in $S$, then $K$ is a union of leaves of the Levi foliation.\end{lemma}
\begin{proof}We start by proving a local version.
\begin{center}\parbox{0.9\textwidth}{\it Let $L$ be a leaf of the Levi foliation. Then for every $p\in L$ there exists a neighborhood $B_p$ such that  every connected component of $L\cap B_p$ is either contained in $K$ or disjoint from $K$.}
\end{center}
Since $S$ is regular and real analytic, there are holomorphic coordinates $(z,w)$ in some neighborhood $V$ of $p$ such that $p$ has coordinates $(0,0)$; we define $B_p=\{(z,w)\ :\ |z|^2+|w|^2< 1\}$, then $S\cap B_p\subseteq \{\Im w=0\}$. The foliation of $S$ in $\{\Im w=0\}\cap\{|z|^2+|w|^2< 1\}$ are the discs $\{\Re w=\mathrm{const}\}\cap\{\Im w=0\}\cap\{|z|^2+|w|^2< 1\}$.

If the local version is false in $B_p$, then $K$ intersects the disc
$$\{\Re w=t_0\}\cap\{\Im w=0\}\cap\{|z|^2+|w|^2<1\}$$
but does not contain it. The set $K_0=K\cap B_p\cap\{\Re w=t_0, \Im w=0\}$ is  a proper relatively closed subset of a complex disc, so it must have a peak point in the disc. That is, there are a complex number $z_0$ and a polynomial $f(z)$ such that  $f(z_0)=1$ and $|f(z)|<1$ if $(z,t_0)\in K_0$ and $z\neq z_0$.Let now
$$\psi(z,w)=|f(z)|-\Re[(w-t_0)^2]\;.$$
Then $\psi$ is a plurisubharmonic function on $B_p$ such that $\psi<1$ on $K_0\setminus\{(z_0,t_0)\}$ and $\psi(z_0,t_0)=1$; this contradicts the local maximum property of $K$, thus proving the local version.

\medskip

Let us now turn our attention to the original statement. Let $L$ be a leaf which intersects $K$; we will denote by $L_{\mathrm{st}}$ the leaf $L$ with the topology induced by the inclusion in $S$ and by $L_{\mathrm{cov}}$ the leaf $L$ with the topology whose open sets are the connected components of the intersection of $L$ with any open set of $S$.

The identity map $id_L:L_{\mathrm{cov}}\to L_{\mathrm{st}}$ is continuous, so $K\cap L$ is a relatively closed subset of both; the local version implies that $K\cap L$ is relatively open in $L_{\mathrm{cov}}$; as $L$ is connected with respect to both topologies, $K\cap L=L$, i.e. $L\subseteq K$.
\end{proof}

As a consequence, we finally obtain the analogue of Proposition \ref{prp_ovvia} for compact Levi-flat hypersurfaces with dense leaves.

\begin{corol}\label{cor_denso}If $M$ is a Grauert type surface and $\chi$ is the pluriharmonic function given by \cite{mst}*{Main Theorem}, then $S=\chi^{-1}(d)$, where $d$ is a regular value, does not contain any proper compact local maximum set.\end{corol}
\begin{proof} By the previous lemma, a local maximum set in $S$ would contain a leaf of the Levi foliation, but every leaf is dense in $S$ (see \cite{mst}*{Corollary 3.8}), so every local maximum set contained in $S$ coincides with $S$.\end{proof}

%

\section{Subdomains of a weakly complete space}

We now prove Theorem \ref{teo_main}.

\medskip

Let $X$ be a complex surface which admits a real-analytic plurisubharmonic exhaustion function and $D\subseteq X$ a domain which admits a $\Ci^0$ plurisubharmonic exhaustion function.

According to \cite{mst}*{Main Theorem}, there are three possibilities:
\begin{enumerate}
\item $X$ is a modification of a Stein space
\item $X$ is proper over an open Riemann surface
\item $X$ is of Grauert type.\end{enumerate}

The first case is the easiest.

\begin{propos} If $X$ is a modification of a Stein space, then $D$ is a modification of a Stein space as well.\end{propos}
\begin{proof} We have that $\Sigma_D\subseteq\Sigma_X$, so $\Sigma_D$ is contained in a union of compact complex curves. By Proposition \ref{prp_ovvia}, $\Sigma_D$ is the union of some irreducible components of these curves, which are then contained in $D$. By Corollary \ref{cor_modstein}, $D$ is a modification of a Stein space.\end{proof}

In the other two cases, we have a set $T\subset\R^n$ and a real-analytic proper map $F:X\to T$ with pluri(sub)harmonic components; for the Grauert type surface it is obvious, with $n=1$ and $F=\chi$ (up to passing to a double cover of $X$), where $\chi$ is the pluriharmonic function given by \cite{mst}*{Main Theorem} (on $X$ or on a double cover of $X$). In the remaining case, let $\pi:X\to R$ be the proper surjective holomorphic map, where $R$ is an open Riemann surface; we embedd properly $R$ in $\C^3$ by $g:R\to\C^3$, then $F=g\circ \pi$ is a proper map from $X$ to $\R^n$, $n=6$, with pluriharmonic components.

We set $V=F(D)$ and we note that in both cases the set of singular values is a discrete set $\{t_j\}_{j\in\N}$. We recall a result from \cite{M}.

\begin{propos}If there is $t_0\in V$ such that $F^{-1}(t_0)\subset D$, then $D=F^{-1}(V)$, hence $F:D\to V$ is proper and $D$ is of the same type as $X$.\end{propos}

To conclude, we just need to prove the following.

\begin{lemma} If for all $t\in V$, $F^{-1}(t)\not\subset D$, then $D$ is a modification of a Stein space.\end{lemma}
\begin{proof}Let $\phi$ be a minimal continuous function for $\Sigma_D$; consider $t_0\in V$ such that $\Sigma_D\cap F^{-1}(t_0)\neq\emptyset$. For $p$ in this intersection, let
$$K=\{F=t_0\}\cap\{\phi=\phi(p)\}\cap\Sigma_D\;.$$
We claim that $K$ is a compact set with the local maximum property; indeed, as $\phi$ and the components of $F$ are plurisubharmonic functions, we can apply \cite{Sl3}*{Lemma 4.8}, which proves that $K$ is a local maximum set. Compactness is obvious.

Either by Proposition \ref{prp_ovvia} or Corollary \ref{cor_denso}, if $t_0$ is a regular value, then $K=F^{-1}(t_0)$, which is impossible; then, $\Sigma_D\cap\{F^{-1}(t)\}\neq \emptyset$ if and only if $t$ is a singular value. Therefore,
$$\Sigma_D\subseteq \bigcup_{j\in\N}F^{-1}(t_j)\;.$$
For $j\in\N$, let
$$K_j=\Sigma_D\cap\{F=t_j\}$$
and note that
$$K_j=\bigcup_{c\in \phi(K_j)}\Sigma_D\cap\{F=t_j\}\cap\{\phi=c\}\;.$$
The sets on the right were shown to be local maximum sets, so their union $K_j$, being closed, also enjoys the local maximum property (see \cite{Sl2}*{Proposition 3.5(b)}).

Fix $j\in N$, let $B$ be a small ball around $t_j$ in $V\subset\R^n$ such that, for $m\neq j$, $t_m\not\in B$ and define $D_j=D\cap F^{-1}(B)$. Consider $v:B\to\R$ a convex exhaustion, then $\rho=\phi+v\circ F$ is a continuous plurisubharmonic exhaustion for $D_j$, which is strongly plurisubharmonic outside $K_j$ (as $\phi$ is minimal for $\Sigma_D$ and $D_j\cap\Sigma_D=K_j$).

By Lemma \ref{lmm_complete}, $K_j$ is contained in the union of finitely many compact curves; therefore $\Sigma_D\subseteq \bigcup K_j$ is contained in the union of at most countably many compact curves. By Corollary \ref{cor_modstein}, $D$ is a modification of a Stein space.
\end{proof}

This concludes the proof of our main result.

\medskip

There are two easy consequences, which are nonetheless worth mentioning.

The first consequence is that, in the situation we study, the minimal kernel does not depend on the regularity of the functions we consider.

\begin{corol}\label{cor_reg_ker}Let $X$ be a complex surface which admits a real-analytic plurisubharmonic exhaustion function and consider a domain $D\subseteq X$ which admits a continuous plurisubharmonic exhaustion function. Then the minimal kernels of $D$ with respect to continuous function and with respect to real-analytic functions are the same set.\end{corol}

The second consequence can be formulated as a regularization result: if we know that a continuous exhaustion exists, we can produce a real-analytic exhaustion.

\begin{corol}\label{cor_reg_exh}Let $X$ be a complex surface which admits a real-analytic plurisubharmonic exhaustion function and consider a domain $D\subseteq X$; if $D$ admits a continuous plurisubharmonic exhaustion function, then it admits a real-analytic plurisubharmonic exhaustion function, which is strongly plurisubharmonic outside the minimal kernel.\end{corol}

Finally, by Corollary \ref{cor_reg_ker}, a minimal continuous exhaustion and a minimal real-analytic exhaustion are strongly plurisubharmonic on the same open set.

It is also worth noting that the results presented imply that we can generalize Corollary 3.1 in \cite{M} to the case of domains with a continuous exhaustion, thus implying an analogue of Main Theorem in \cite{mie} and Theorem 1.1 in \cite{LY} for such domains in a Hopf surface.

\end{document}